\documentclass[reqno]{article}
\usepackage{amsmath, amsfonts, amsthm, amssymb, eucal, indentfirst}

\newtheorem{theorem}{Theorem}
\newtheorem{lemma}[theorem]{Lemma}

\begin{document}

\title{An evasion game on a graph}
\author{John Haslegrave \thanks{School of Mathematics and Statistics, University of Sheffield, Sheffield, UK. {\tt j.haslegrave@cantab.net}}}
\maketitle

\begin{abstract}
This paper introduces a pursuit and evasion game to be played on a connected graph. One player moves invisibly around the graph, and the other player must guess his position. At each time step the second player guesses a vertex, winning if it is the current location of the first player; if not the first player must move along an edge. It is shown that the graphs on which the second player can guarantee to win are precisely the trees that do not contain a particular forbidden subgraph, and best possible capture times on such graphs are obtained.
\end{abstract}

\section{Introduction}

Pursuit and evasion games on graphs have been widely studied. Perhaps the most significant is the Cops and Robbers game, an instance of which is a graph $G$ together with a fixed number of cops. The cops take up positions on vertices of $G$ and a robber then starts on any unoccupied vertex. The cops and the robber take turns: the robber chooses either to remain at his current vertex or to move to any adjacent vertex, and then the cops simultaneously make moves of the same form. The game is played with perfect information, so that at any time each of the players knows the location of all others. The cops win if at any point one of them is at the same position as the robber. Early results on this game include those obtained by Nowakowski and Winkler \cite{NW83} and Aigner and Fromme \cite{AF84}. An important open problem is Meyniel's conjecture, published by Frankl \cite{Fra87}, that $O(\sqrt{n})$ cops are enough to win on any $n$-vertex connected graph. More recently, several variations on the game have been analysed by Clarke and Nowakowski (e.g. \cite{CN00}).

In this paper we will consider a novel form of pursuit game, which bears some similarity to the Cops and Robbers game but differs in that the movement of the pursuer (the cat) is not constrained by the edges of the graph, and also in that the pursuer is disadvantaged by not knowing where the pursued (the mouse) is. This imperfect information will naturally lead to a different emphasis: we ask whether there is a strategy for the cat that is successful against any possible strategy for the mouse, and if so how long it takes. Another recent variation of a similar type is the Robber Locating game, introduced by Seager \cite{Sea12} and further studied by Carraher, Choi, Delcourt, Erickson, and West \cite{CCDEW}, in which a cop probes a vertex at each turn and is told the current distance to the robber.

Descriptively, we consider a connected graph to represent a network of mouse-holes connected by passageways. The cat tries to catch the mouse by inserting a paw into one of the holes; if the cat has chosen the correct hole, then the mouse is caught. After each unsuccessful attempt, the mouse moves from the hole he is currently in to any adjacent hole. A rudimentary form of this problem, asking how the cat may win on a path, appeared on an internet puzzle forum \cite{xkcd}. 

We consider the active version, in which the mouse is required to move. The cat cannot guarantee to win without this restriction (on at least two vertices), since the mouse would have at least two options at each time step, one of which avoids the cat. The active game is not feasible if there is only one vertex; clearly the cat can guarantee to catch the mouse in two attempts (by choosing the same vertex twice) on the two-vertex connected graph, and cannot do better, so we shall subsequently assume that our graph has at least three vertices.

\section{Strategies for the mouse}

We say that the mouse can survive to time $t$ on a graph $G$ if, for any sequence $c_1, \ldots, c_t$ of vertices of $G$, there exists a sequence $m_1, \ldots, m_t$ of vertices such that $m_i\neq c_i$ for every $i$, and $m_im_{i+1}$ is an edge of $G$ for $1\le i \le t-1$. We will refer to the sequence $(c_i)$ as a \textit{cat sequence} and $(m_i)$ as a \textit{mouse sequence} that beats it. For each $G$ we wish to determine whether there is a $t$ such that the mouse cannot survive to time $t$, and to determine the least such $t$ if so. Write $m(G)$ for the least such $t$, if it exists, and $m(G)=\infty$ otherwise, so that the mouse can survive to time $t$ on $G$ if and only if $m(G)>t$. The main aim of this paper is to find a necessary and sufficient condition on $G$ for $m(G)$ to be finite, and a simple formula for $m(G)$ if $G$ is such a graph.

Consider first the case where $G$ is a cycle. In this case the mouse may always survive, because at every stage he has a choice of two moves and at least one of them must be safe. Formally, given a cat sequence $c_1,\ldots,c_t$ we may inductively find a mouse sequence that beats it: take a mouse sequence $m_1,\ldots, m_{t-1}$ to beat $c_1,\ldots,c_{t-1}$ and choose for $m_t$ a neighbour of $m_{t-1}$ that is not equal to $c_t$; this is possible since there are two neighbours to choose from.

Trivially, if $H$ is a subgraph of $G$ and the mouse can survive to time $t$ on $H$ then he can survive to time $t$ on $G$ by restricting himself to making moves on $H$. Consequently the mouse can always survive if $G$ contains a cycle. If $m(G)<\infty$, then, $G$ must be a tree.

Next we shall show that there are some trees on which the mouse can always survive. Let $T^*$ be the tree consisting of three paths of length 3 with one common endpoint, with the $j$th path having vertices $u_j$, $v_j$, $w_j$ and $x$ in that order.

\begin{lemma}\label{tstar}The mouse can survive to time $t$ on $T^*$ for any $t$.
\end{lemma}
\begin{proof}
Given a cat sequence $c_1, \ldots, c_t$ on $T^*$, we shall construct a mouse sequence $m_1, \ldots, m_t$ that beats it. The key idea is that at every odd time the mouse
will be at $x$ or at distance 2 from $x$, and it will be $x$ whenever possible. We shall show that when the mouse is forced away from $x$ he may choose one of the three arms of $T^*$ to move down in such a way that he will be able to return to $x$ once it is safe to go back there (though he may need to know the cat sequence arbitrarily far in advance in order to make the correct choice).

It suffices to prove the assertion for odd $t$, since when $s$ is even we shall then have proved that the mouse can survive to time $s+1$, and consequently to time $s$. The argument does not depend on $t$ being odd, but we construct even terms of the mouse sequence from the neighbouring odd terms, so this will avoid having a separate case for the final term.

Set $m_i=x$ for every odd $i$ such that $c_i\neq x$. We must now choose a suitable value of $m_i$ for every odd $i$ with $c_i=x$. For each such $i$, we must have $m_i=u_j$, $m_i=v_j$, or $m_i=w_j$ for some $j$. We divide such $i$ into maximal subsequences of consecutive odd terms of the cat sequence taking the value $x$. Within each such group we ensure that we consistently choose the same value of $j$, since the mouse must go down one arm and may not return to $x$ for the duration of this group. For each $i$ and $k$, with $i$ odd, such that $c_i=c_{i+2}=\cdots=c_{i+2k}=x$ but $c_{i-2}\neq x$ (or $i=1$) and $c_{i+2k+2}\neq x$ (or $i+2k=t$) we choose $j\in \{1,2,3\}$ such that $c_{i-1}\neq w_j$ and $c_{i+2k+1}\neq w_j$. This is possible since at most two out of the three values are not permitted. Now set $m_i=m_{i+2}=\cdots=m_{i+2k}=v_j$. This choice ensures that the mouse can safely move from $x$ to $v_j$ and back again when required.

We have now defined $m_i$ for all odd $i$. If $i$ is even and $m_{i-1}=m_{i+1}=x$, then choose any $j$ with $c_i\neq w_j$ and set $m_i=w_j$. If $i$ is even and $m_{i-1}=m_{i+1}=v_j$ for some $j$, then set $m_i=u_j$ or $m_i=w_j$, whichever is not equal to $c_i$. By our construction of $m_i$ for odd $i$, the only other possibility for even $i$ is that one of $m_{i-1}$ and $m_{i+1}$ is $x$ but the other is $v_j$ for some $j$; in that case either $c_{i-1}=x\neq c_{i+1}$ or $c_{i+1}=x\neq c_{i-1}$. Our choice of odd mouse values then implies that $c_i\neq w_j$ and we can take $m_i=w_j$. In every case we have chosen $m_i$ for even $i$ to be adjacent to $m_{i-1}$ and $m_{i+1}$, and $m_i\neq c_i$ for all $i$, as required.
\end{proof}

In fact $T^*$ is essentially the only example of a tree on which the mouse can survive: we shall show that the cat can always catch the mouse on any tree that does not have $T^*$ as a subgraph. We refer to such trees as $T^*$-{\em free}.

Before giving a strategy for the cat to win on any $T^*$-free tree $T$ with at least three vertices, we prove a lower bound on $m(T)$. We will show later that this bound is equal to $m(T)$ when $T$ is $T^*$-free. The key idea in defining the lower bound is to consider how often the cat must visit each vertex.

Let $T$ be a tree with at least three vertices, and let $v$ be a vertex of $T$. Define $\tilde{d}(v)$ as the number of neighbours of $v$ that are not leaves. Define $a(v)$ for $v\in V(T)$ as follows:
\[
a(v)=
\begin{cases}
2\tilde{d}(v)-2 &\text{if } \tilde{d}(v)\ge 2\,; \\
2 &\text{if } d(v)\ge 2 \text{ but } \tilde{d}(v)<2\,; \\
0 &\text{if } d(v)=\tilde{d}(v)=1\,.
\end{cases}
\]
Let $A(T)=\sum_v a(v)$. If $d(v)=1$ but $\tilde{d}(v)=0$, then $T$ is the two-vertex tree, which we have already excluded.

\begin{lemma}\label{cmlb}If $T$ is a tree with at least three vertices, then $m(T)\ge A(T)$.
\end{lemma}
\begin{proof}
We shall prove the stronger statement that, for each vertex $v$, an unbeatable cat sequence must visit $v$ at least $a(v)$ times. An unbeatable sequence must therefore have length at least $\sum_v a(v)=A(T)$.

Let $c_1,\ldots,c_t$ be a cat sequence that visits some vertex $v$ at most $a(v)-1$ times; we aim to construct a mouse sequence that beats it. Since $a(v)\ge 1$ (in fact $a(v)\ge 2$ since it is never defined to be 1), $d(v)\ge 2$. Our strategy is for the mouse to stay at or adjacent to $v$ as much as possible. As in the proof of Lemma \ref{tstar}, we must then show that when he is forced to leave he can choose a direction that will allow him to remain safe until it is possible to return. We distinguish two cases according to the value of $a(v)$.

\vspace{5pt}
\textbf{Case 1.} $a(v)=2$.

In this case there is a winning mouse sequence that always stays at or adjacent to $v$. Since $c_i=v$ for at most one value of $i$, either $c_i\neq v$ for all odd $i$ or $c_i\neq v$ for all even $i$. Assume without loss of generality the former, and let $m_i=v$ for every odd $i$. Let $u$ and $w$ be two neighbours of $v$. For each even $i$, let $m_i=u$ if $c_i\neq u$ and $m_i=w$ if $c_i=u$. This is a valid mouse sequence and $m_i\neq c_i$ for each $i$, as required.

\vspace{5pt}
\textbf{Case 2.} $a(v)>2$.

In this case $\tilde{d}(v)>2$.  Write $r=\tilde{d}(v)$ and let $u_1,\ldots,u_r$ be the internal neighbours of $v$; for each $1\le j\le r$ let $w_j$ be a neighbour of $u_j$ not equal to $v$. Again we try to stay at or adjacent to $v$. When we are forced away from it we will choose one of the $w_j$ to move to; we need to show that there is one we can move to safely. 

Since $a(v)-1=2r-3$, the sequence $(c_i)$ visits $v$ at most $2r-3$ times. Hence either $c_i=v$ for at most $r-2$ odd values of $i$ or $c_i=v$ for at most $r-2$ even values of $i$; assume without loss of generality the former. Let $m_i=v$ for every odd $i$ with $c_i\neq v$. Next we choose a suitable value of $m_i$ for every odd $i$ with $c_i=v$. We divide such $i$ into maximal subsequences of consecutive odd terms of the cat sequence taking the value $v$. Within each such group we consistently choose the same vertex. For each $i$ and $k$, with $i$ odd, such that $c_i=c_{i+2}=\cdots=c_{i+2k}=v$ but $c_{i-2}\neq v$ (or $i=1$) and $c_{i+2k+2}\neq v$ (or $i+2k+2>t$), we choose $j$ with $1\le j\le r$ such that $c_{i-1},c_{i+1},\ldots,c_{i+2k+1}\neq u_j$. This is possible since $k\le r-3$ by assumption, and so at most $r-1$ of the possible $j$ are excluded by our condition. Now set $m_i=m_{i+2}=\cdots=m_{i+2k}=w_j$. 

We have now defined $m_i$ for all odd $i$. If $i$ is even and $m_{i-1}=m_{i+1}=v$ (or $m_{i-1}=v$ and $i+1>t$), then choose any $j$ with $c_i\neq u_j$ and set $m_i=u_j$. If $i$ is even and $m_{i-1}=w_j$ or $m_{i+1}=w_j$ for some $j$, then set $m_i=u_j$. This is consistent, since if $m_{i-1}=w_j$ and $m_{i+1}=w_l$ then, by choice of $m_i$ for odd $i$, $c_{i-1}=c_{i+1}=v$ and so $j=l$. This is a valid mouse sequence and $m_i\neq c_i$ for each $i$, as required.
\end{proof}

We now observe that in certain cases removing a leaf from $G$ does not change $m(G)$, which will allow us to reduce any given tree to a tree of a certain form on which the mouse can survive for the same length of time. 

\begin{lemma}\label{prune}Let $G$ be a tree and $u$ be a leaf whose neighbour, $v$, has degree at least 3. If $H$ is the tree obtained from $G$ by deleting $u$, then $m(H)=m(G)$.
\end{lemma}
\begin{proof}
Certainly $m(H)\le m(G)$, since $H$ is a subgraph of $G$. Now suppose that the mouse can survive to time $t$ on $G$, and let $c_1,\ldots,c_t$ be a cat sequence on $H$. It is also a cat sequence on $G$, so there are mouse sequences on $G$ that beat it; take $m_1,\ldots,m_t$ to be such a mouse sequence with the fewest occurrences of $u$. If $m_i\neq u$ for every $i$, then $m_1,\ldots,m_t$ is also a mouse sequence on $H$, so we are done. If $m_i=u$, then $m_{i-1}=m_{i+1}=v$ since $v$ is the only neighbour of $u$. Note that $v$ has at least two other neighbours, at least one of which is not $c_i$; let $w$ be such a neighbour. The sequence given by $m'_j=m_j$ for $j\neq i$ and $m'_i=w$ is also a mouse sequence that beats $c_1,\ldots,c_t$, but it has fewer occurrences of $u$, contradicting minimality. Thus also $m(H)\ge m(G)$.
\end{proof}

Starting from a given tree $T_0$ (with at least three vertices and so at least one internal vertex), define a sequence $T_0,T_1,\ldots$ by removing a leaf of $T_i$ adjacent to a vertex of degree at least 3 to obtain $T_{i+1}$. Do this until no such leaves remain in $T_r$. Lemma \ref{prune} yields $m(T_0)=m(T_1)=\cdots=m(T_r)$. Since we never remove a neighbour of a vertex of degree 2, we do not create any new leaves by this process. Hence all internal vertices of $T_0$ remain internal in $T_r$, and thus the neighbour of any leaf in $T_r$ has degree 2. Let us call such a tree a \textit{pruned tree}, and define a \textit{twig} to be a vertex $v$ such that $d(v)\ge 2$ but at least $d(v)-1$ of its neighbours are leaves. In a pruned tree that is not a star, every twig has degree 2, and no vertex other than a twig has a leaf neighbour. Thus if $T_0$ is not a star and has $n$ vertices, $t$ twigs and $l$ leaves, then the resulting pruned tree has $n+t-l$ vertices.

\section{Strategies for the cat}

We shall now show how to construct a winning strategy for the cat on any $T^*$-free tree, thus obtaining an upper bound for $m(T)$ for each such $T$. The two bounds coincide: the strategy we construct for $T$ takes time $A(T)$, so the bound in Lemma \ref{cmlb} shows that our strategy is optimal. We first find a more convenient condition that is equivalent to being $T^*$-free.

\begin{lemma}\label{cmpath}A tree $T$ is $T^*$-free if and only if it contains a path $P$ such the maximum distance of a vertex from $P$ is at most 2.
\end{lemma}
\begin{proof}
We prove the contrapositive for both implications.  If $T^*\subseteq T$, then for any path $P$ in $T$, some leaf of $T^*$ has distance at least 3 from $P$.  Conversely, let $P$ be a longest path in $T$, and let $w$ be a vertex having distance 3 from $P$. Let $x$ be the vertex where the path from $w$ meets $P$.  Since $P$ is a longest path, it extends at least three vertices in each direction from $x$. Together with the path from $w$ to $x$, this subpath forms $T^*$.
\end{proof}

\begin{lemma}\label{cmub}If $T$ is a $T^*$-free tree with at least three vertices, then $m(T)\le A(T)$.
\end{lemma}
\begin{proof}
We will first choose a path as guaranteed Lemma \ref{cmpath}. The cat will move along this path, checking each vertex and its internal neighbours before moving on. We will consider the parity of the mouse's starting position, and show that we catch a mouse of the right parity by doing this; we catch a mouse of the other parity by a second pass of the same form back along the path.

To make the strategy easier to define, we wish to choose our path so that it contains no twig. We show first that this is possible for all but a small class of graphs; we deal with this class separately.

Let $P$ with vertices $v_1,\ldots,v_r$ be a shortest path among those having the property that all vertices of $T$ are within distance 2 of the path. If $P$ contains more than one vertex, then there must exist a vertex $x$ with $d(x,v_1)=2$ and $d(x,v_i)>2$ for $i\neq 1$, for otherwise $P-v_1$ is a shorter path with the required property. Consequently the common neighbour of $x$ and $v_1$ is an internal vertex not equal to $v_2$, and $v_1$ is itself an internal vertex. Likewise $v_r$ has an internal neighbour not equal to $v_{r-1}$. Since each vertex on the path is now an internal vertex, $\tilde{d}(v_i)\ge 2$ for each $i$. 

If $P$ consists of one vertex $v$, then $\tilde{d}(v)\ge 2$, $\tilde{d}(v)=0$, $\tilde{d}(v)=d(v)=1$, or $\tilde{d}(v)=1<d(v)$. In the first case $P$ has no twig, as desired. In the second case $T$ is a star with central vertex $v$. In the third case it is a star whose central vertex is the neighbour of $v$. In the fourth case $T$ is the \textit{double-star} consisting of two adjacent internal vertices $v$ and $w$ and some leaves. If $T$ is a star, then $m(T)=2\ge A(T)$ since the mouse must occupy the central vertex either at time 1 or time 2, so the cat can win by choosing the central vertex twice. If $T$ is a double-star with internal vertices $v$ and $w$, then $A(T)=4$; we claim that the cat sequence $v,w,w,v$ is unbeatable. Suppose not and let $m_1,m_2,m_3,m_4$ be a mouse sequence that beats it. Since $c_1=v$ and $c_2=c_3=w$, we have $m_1$ not a neighbour of $w$, and $m_1=w$ implies $m_2=v$.  Hence in each case $m_2=v$, and then $c_3=w$ implies that $m_3$ is a leaf neighbour of $v$ and the cat wins at time 4.

If $T$ is not a star or a double-star, then either $P$ has at least two vertices or has one vertex $v$ with $\tilde{d}(v)\ge 2$. In either case we have $\tilde{d}(v_i)\ge 2$ for $1\le i\le r$. Also, if $w$ is not on $P$, then $\tilde{d}(w)<2$, since otherwise $w$ has some internal neighbour $x$ with $d(x,P)=d(w,P)+1$. Similarly, $x$ has a neighbour $y$ other than $w$, and $d(y,P)=d(x,P)+1=d(w,P)+2>2$, contradicting the choice of $P$.

For such trees, the cat plays as follows. For each $i$, write $b_i=\tilde{d}(v_i)$ and let $w_{i,1},w_{i,2},\ldots,w_{i,b_i}$ be the internal neighbours of $v_i$, ordered with $v_{i-1}$ first (if $i\neq 1$) and $v_{i+1}$ last (if $i\neq r$). Let $B_1=2$ and $B_i=2+\sum_{j<i}(2b_j-3)$ for $i>1$. For each $i$, let $c_{B_i}=c_{B_i+2}=\cdots=c_{B_i+2b_i-4}$, and let $c_{B_i+2k-3}=w_{i,k}$ for $1\le k\le b_i$. Note that, since $B_i+2b_i-4=B_{i+1}-1$ and by our choice of order, which specified $w_{i,1}=v_{i-1}$ (if $i\neq 1$) and $w_{i,b_i}=v_{i+1}$ (if $i\neq r$), this definition is consistent and defines $c_s$ for $1\le s\le t$ where $t=B_r+2b_r-3$; let $c_s=c_{2t+1-s}$ for $t+1\le s \le 2t$. In the cat sequence $c_1,\ldots,c_{2t}$, each vertex $v_i$ on the path is visited $2\tilde{d}(v_i)-2$ times, and each internal vertex not on the path is visited twice, so $2t=A(T)$. 

We claim that this sequence is unbeatable. Since the mouse must always move to a neighbouring vertex, its distance from $v_1$ must change by 1 at each time step. Consequently, either $d(m_s,v_1)$ has the same parity as $s$ for every $s$, or $d(m_s,v_1)$ has opposite parity to $s$ for every $s$. We distinguish two cases according to the parity of the mouse.

\vspace{5pt}
\textbf{Case 1.} $d(m_s,v_1)$ has the same parity as $s$ for each $s$.

Our aim is to show that parity considerations force the mouse to stay on the same side of the cat to avoid being caught, so that as the cat sequence moves along the path the mouse will be forced into a successively smaller region until he is cornered.

Note that $d(c_s,v_1)$ has the same parity as $s$ for $1\le s \le t$, so $d(m_s,c_s)$ is even for $1\le s \le t$. $c_s$ first takes the value $v_i$ when $s=B_i$ and last takes this value (in the range $1\le s \le t$) when $s=B_i+2b_i-4$. Consider the components of $T-v_i$ containing $m_s$ and $m_{s'}$ for $B_i-1\le s\le B_i+2b_i-3$ and $B_i-1\le s'\le B_i+2b_i-3$. If these are different (or either $m_s$ or $m_{s'}$ is $v_i$), then $m_{s''}=v_i$ for some ${s''}$ such that $B_i-1\le s''\le B_i+2b_i-3$, but the specified cat sequence yields $d(c_{s''},v_i)\le 1$ and so $d(c_{s''},m_{s''})\le 1$. Since $d(c_{s''},m_{s''})$ is even, the mouse is caught. 

We have shown that in order to avoid capture the mouse must be in some component $C_i$ of $T-v_i$ whenever $B_i-1\le s\le B_i+2b_i-3$. If $C_i$ is neither the component containing $v_{i-1}$ nor the component containing $v_{i+1}$, then each vertex in $C_i$ is nearer to $v_i$ than to any other vertex of $P$, so each vertex in $C_i$ has distance at most 2 from $v_i$. Since $C_i$ contains only one neighbour $w$ of $v_i$, any two vertices in $C_i$ are adjacent. Since the mouse remains in $C_i$ for at least three consecutive terms of the sequence, $C_i$ contains more than one vertex; hence $w$ is an internal vertex. The distance $d(m_s,c_s)$ is even when $B_i-1\le s\le B_i+2b_i-3$, but for some $s'$ in this range $c_{s'}=w$ (since $c_{B_i-1},\ldots,c_{B_i+2b_i-3}$ are the internal neighbours of $v_i$ in some order), so $d(m_{s'},c_{s'})\le 1$. Since $d(c_{s'},m_{s'})$ is even, the mouse is caught. 

Therefore, if the mouse is not caught, then always $C_i$ contains $v_{i-1}$ or $v_{i+1}$. Since $C_1$ contains $v_2$, and $C_r$ contains $v_{r-1}$, there exists $j$ such that $C_j$ contains $v_{j+1}$ and $C_{j+1}$ contains $v_j$. Now $m_{B_{j+1}}$ is in both the component of $T-v_j$ containing $v_{j+1}$ and the component of $T-v_{j+1}$ containing $v_j$. These two components are disjoint, and a contradiction is obtained.

\vspace{5pt}
\textbf{Case 2.} $d(m_s,v_1)$ has the opposite parity to $s$ for each $s$.

Consider the sequence $(m'_s)$ given by $m'_s=m_{2t+1-s}$. Since $d(m'_s,v_1)$ has the opposite parity to $2t+1-s$, it has the same parity as $s$, so we know from the first case that there exists $s$ with $c_s=m'_s=m_{2t+1-s}$. By construction, $c_s=c_{2t+1-s}$, so the mouse is caught.
\end{proof}

Combining Lemma \ref{cmub} with Lemma \ref{cmlb}, we see that $m(T)=A(T)$ for any $T^*$-free tree with at least three vertices. However, we can simplify the description of $A(T)$ to obtain the following classification.

\begin{theorem}\label{cmdone}Let $T$ be a tree with $n\ge 3$ vertices, $t$ twigs and $l$ leaves. If $T$ contains $T^*$ as a subgraph, then $m(T)=\infty$. If $T$ is a star, then $m(T)=2$. Otherwise, $m(T)=2n+2t-2l-4$.
\end{theorem}
\begin{proof}
Lemma \ref{tstar}, together with the observation that $m(G)\ge m(H)$ if $H$ is a subgraph of $G$, gives $m(T)=\infty$ whenever $T^*$ is a subgraph of $T$. If $T$ is $T^*$-free, then $m(T)=A(T)$ by Lemmas \ref{cmub} and \ref{cmlb}. If $T$ is a star, then $A(T)=2$, as required. Our remarks following Lemma \ref{prune} showed that we may find a pruned tree $S$ with $m(T)=m(S)$ by deleting leaves adjacent to vertices of degree at least three from $T$. Since $S$ is a subgraph of $T$, it is also $T^*$-free (and has at least three vertices), so $m(T)=m(S)=A(S)$.

If $S$ is a pruned tree with $k$ vertices and $v$ is a vertex of $S$ with $d(v)\ge 3$, then $v$ has no leaves as neighbours, so $d(v)=\tilde{d}(v)$. Consequently, if $d(v)\ge 3$ or $d(v)=\tilde{d}(v)=2$, then $a(v)=2\tilde{d}(v)-2=2d(v)-2$. If $d(v)=2$ and $\tilde{d}(v)<2$, then $a(v)=2=2d(v)-2$. If $d(v)=1$, then $a(v)=0=2d(v)-2$. So $a(v)=2d(v)-2$ for each vertex $v$. Consequently, $A(S)=\sum_va(v)=2\sum_vd(v)-2k$. Since $S$ is a tree, $\sum_vd(v)=2k-2$ and $A(S)=2k-4$. As noted earlier, $S$ has $n+t-l$ vertices when $T$ has $n$ vertices, $t$ twigs, and $l$ leaves, so $m(T)=A(S)=2n+2t-2l-4$, as required.
\end{proof}

\section{Concluding remarks}

Theorem \ref{cmdone} completely solves the problem for a single cat and mouse. A natural extension would be to ask for a classification of graphs on which two (or $k$) cats can co-operate to catch the mouse.

\end{document}